\documentclass[11pt, final]{article}
\usepackage[a4paper, margin = 1in]{geometry}

\usepackage{amsmath, amssymb, amsthm, amsbsy}
\usepackage{mathtools, bbm, dsfont}
\usepackage{enumitem}
\usepackage{microtype}
\usepackage[indent]{parskip}
\usepackage{graphicx} 
\usepackage{xcolor}
\usepackage{mathrsfs}
\usepackage{tikz,caption, subcaption}
\usetikzlibrary{shapes,arrows,math, calc, positioning, fit, shapes.misc, hobby}

\usepackage[pdftex,colorlinks,citecolor=blue]{hyperref}
\usepackage{cleveref}
\usepackage[numbers,sort&compress]{natbib}
\usepackage[notcite,notref]{showkeys}
\def\mcite[#1]#2{\mbox{\cite[#1]{#2}}} 

\newtheorem{theorem}{Theorem}[section]
\newtheorem{lemma}[theorem]{Lemma}

\theoremstyle{definition}
\newtheorem{claim}{Claim}[theorem]

\newtheorem*{claim*}{Claim}
\newtheorem*{remark*}{Remark}

\newcommand{\vv}{\textbf{v}}

\newcommand{\oldqed}{}
\def\endofClaim{\hfill\scalebox{.6}{$\Box$}}
\newenvironment{claimproof}[1][Proof]{
  \renewcommand{\oldqed}{\qedsymbol}
  \renewcommand{\qedsymbol}{\endofClaim}
  \begin{proof}[#1]}
  {\end{proof}
  \renewcommand{\qedsymbol}{\oldqed}}

\definecolor{andrey}{rgb}{0.3, 0.65,,0.2}

\title{Range of Clique Counts in Graphs}
\author{%
Mihir Neve\thanks{Department of Mathematics, The London School of
Economics, UK. E-mail: {\tt  m.s.neve@lse.ac.uk}.}
\and
Alexey Pokrovskiy\thanks{Department of Mathematics, University College London, UK. Email: {\tt a.pokrovskiy@ucl.ac.uk}.}
\and
Andrey Shapiro \thanks{Department of Mathematics, King’s College London, UK. E-mail: {\tt andrey.shapiro@kcl.ac.uk}.}
}

\date{}

\begin{document}

\maketitle

\begin{abstract}
    Let $\gamma(G)$ denote the number of cliques in a graph $G$ and let $\Gamma(n)\coloneqq \{\gamma(G):|V(G)|=n\}$ be the set of values of $\gamma(G)$ that can be attained on $n$ vertices. We improve on a result by Erd\H{o}s and Ern\'e to show that $| \Gamma(n)| \geq 2^{n-4\ln(2)\log^3(n)}$ for sufficiently large $n$.
\end{abstract}

\section{Introduction}

Enumeration of graph substructures has been a fundamental problem of structural and algorithmic interest in combinatorics. An early result in this area includes, for instance, Cayley's theorem~\cite{borchardt1860ueber, cayley1889theorem}, which establishes that the number of spanning trees in the $n$-vertex complete graph~$K_n$ is exactly~$n^{n-2}$. This result was later generalised to arbitrary graphs by Kirchhoff~\cite{kirchhoff1847ueber}, who linked spanning tree counts to the determinant of a specific symmetric matrix derived from the underlying graph structure. 

Beyond spanning trees, the enumeration of  homogeneous substructures, such as cliques and independent sets, has also been intensively studied. Recall that a \emph{clique} in a graph~$G$ is a (not necessarily maximal) subset of vertices~$S \subseteq V(G)$ that induces a complete subgraph~$K_{|S|}$ in~$G$. Notably, the enumeration of several combinatorial structures can be reduced to counting cliques, or equivalently independent sets, in some aptly constructed auxiliary graph. For instance, independent sets in Cayley graphs have been used to enumerate sum-free sets in Abelian groups (see, for example, \cite{sum_free_sets_Alon}). Similarly, Kleitman and Winston~\cite{C_4free, Lattices} estimated the number of lattices and $C_4$-free graphs by estimating the number of independent sets in some auxiliary graphs. In fact, their algorithmic approach gave rise to the so-called container method, which has since become a powerful tool for estimating the number of cliques and independent sets in a wide class of graphs (see, for example,~\cite{container_ind_sets} for a comprehensive survey).  

Several variations of the clique counting problem have been explored in the literature. 
As a generalisation to Tur\'{a}n's theorem, Zykov~\cite{zykov} determined the maximum number of $k$-cliques in an $n$-vertex graph having no $\ell$-cliques for all $k < \ell$, which has since been generalised in several directions (see, for example, \cite{Eckhoff, erdos_circuit, Fisher_Ryan}). More recently, Wood~\cite{wood_cliques} established tight bounds on the maximum number of cliques in an $n$-vertex graph with $m$~edges, further extending these bounds to the graph classes of planar graphs, graphs with bounded degree, and graphs with bounded degeneracy. Focusing on maximal cliques, Moon and Moser~\cite{moon_moser} determined the maximum number of maximal cliques in an $n$-vertex graph, and further characterised the extremal graphs. They also estimated the number of different sizes of maximal cliques achievable by an $n$-vertex graph. Their bounds were subsequently improved by Erd\H{o}s~\cite{Erdos_cliques} and Spencer~\cite{Spencer_cliques}. 

In this paper, we study the range of values that can be attained by the clique counts of $n$-vertex graphs. More concretely, let~$\gamma(G)$ denote the number of cliques in a graph~$G$, and let~$\Gamma(n) \coloneqq \{\gamma(G) : |V(G)| = n\}$ be the set of possible clique counts achievable by an $n$-vertex graph. Note that for any $n$-vertex graph~$G$, the clique count~$\gamma(G)$ trivially satisfies~$n+1 \leq \gamma(G) \leq 2^n$, where the lower and upper bounds are uniquely attained by the empty graph and the complete graph, respectively.

Erd\H{o}s and Ern\'{e}~\cite{erdos_erne_orig} were interested in the asymptotic behaviour of the size of~$\Gamma(n)$. They established the following bounds. Note that unless specified otherwise, all logarithms throughout this paper are taken to base~$2$.

\begin{theorem}[\cite{erdos_erne_orig}]
\label{thm:erdos_erne}
    For sufficiently large $n$, we have $2^{n-6n^{5/6}}\leq |\Gamma(n)| = o(2^{n - 0.4\log n})$.
\end{theorem}

The lower bound in Theorem~\ref{thm:erdos_erne} was proved by constructing a family of $n$-vertex graphs having pairwise distinct clique counts. In fact, they prove a stronger structural result by showing that~$\Gamma(n)$ contains a contiguous set of integers ranging from $n+1$ up to this exponential lower bound. In contrast, to derive the upper bound, they observe that for graphs where~$\gamma(G)$ is close to $2^n$, the achievable clique counts exhibit discrete gaps. Furthermore, they characterise these sparse upper counts by proving that their binary expansions must contain a bounded number of ones. Factoring this characterisation into the enumeration of possible clique counts yields the upper bound in Theorem~\ref{thm:erdos_erne}.

Recently,  Kov\'{a}cs and Nagy~\cite{count_ind_sets} rediscovered this problem in the guise of counting independent sets --- compared to~\cite{erdos_erne_orig},  they proved  a better lower bound~$\log(|\Gamma(n)|) \geq n - 2^{(1+o(1))\sqrt{\log n}}$ and a weaker upper bound $\log(|\Gamma(n)|) \leq n-0.2075\log n$. In this paper, we improve the lower bounds of~\cite{erdos_erne_orig} and~\cite{count_ind_sets} and obtain something of similar shape to the upper bound. 

\begin{theorem}
\label{thm:main}
    For sufficiently large $n$, we have $|\Gamma(n)| \geq 2^{n-4\ln(2)\log^3(n)}$.
\end{theorem}

\section{Notation and Outline of Proof}

We use~$[n]$ to denote the set of integers $\{1, \ldots, n\}$ for $n \in \mathbb{N}$.  Given a graph~$G$ and a collection of pairwise disjoint subsets~$V_1, \dots, V_k \subseteq V(G)$, a set of vertices~$T \subseteq \cup_{i \in [k]} V_i \subseteq V(G)$ is said to be a \emph{transversal} with respect to~$\{V_i\}_{i \in [k]}$ if~$|T \cap V_i| = 1$ for all~$i \in [k]$. Next, let~$G_1 = (V_1, E_1)$ and~$G_2 = (V_2, E_2)$ be vertex-disjoint graphs containing $k$-cliques~$K_1 \subseteq V_1$ and~$K_2 \subseteq V_2$, respectively. Then, the \emph{clique sum}~$G$ of $G_1$ and $G_2$ is obtained by identifying the cliques~$K_1$ and~$K_2$. More precisely, $G$ has the vertex set $V(G) = (V_1\setminus K_1) \sqcup K \sqcup (V_2\setminus K_2)$, where $K$ forms a $k$-clique, and for each~$i \in \{1,2\}$, we have $G[(V_i\setminus K_i) \cup K] \simeq G_i$, $G[V_i\setminus K_i] \simeq G_i[V_i \setminus K_i]$, and there are no edges between $(V_1\setminus K_1)$ and $(V_2\setminus K_2)$. 

Let $k = n - \lfloor 4 \ln(2) \log^3(n)\rfloor$. The general idea of our proof of Theorem~\ref{thm:main} is to construct a family of $2^{k}$ $n$-vertex graphs having pairwise distinct clique counts. This construction will be done primarily in two steps. First, we construct a family of \emph{base graphs} parametrised by integers~$a$ and~$\ell$. Each graph in this family will contain a clique of size~$k+8$ and a vertex-disjoint balanced $\ell$-partite graph. By systematically varying the edges between the clique and the $\ell$-partite subgraph, we generate a family of base graphs with distinct clique counts. This is outlined in Lemma~\ref{q2} of Section~\ref{sec:Base_graph}.

However, the variations in clique counts achieved by the base graphs alone are of the order~$2^{\Theta(n/ \log n)}$, which is insufficient to prove Theorem~\ref{thm:main}. To overcome this, as a second step, we boost the set of achievable clique counts by taking the clique sum of multiple base graphs. 
Specifically, for each $x \in [2^k]$, we construct an $n$-vertex graph~$G_x$ with a unique clique count as follows. Note that $x-1$ has a binary representation on $k$ digits. By partitioning the index set~$[k]$ into intervals $[a_1, a_2), [a_2, a_3), \dots, [a_w, a_{w+1})$ for an aptly chosen integer $w$, we can uniquely represent $x - 1$ as the sum $x - 1 = \sum_{i \in [w]}2^{a_i}(x_i -1)$ for integers $x_i \in [2^{a_{i+1} - a_i}]$. For each integer~$x_i$, and input parameters $a = a_i$ and $\ell = \log (a_{i+1} - a_i)$, Lemma~\ref{q2} yields a base graph~$H_{x_i}$ containing a clique of size~$k+8$. The final graph~$G_x$ is then obtained by taking successive clique sums of the base graphs~$H_{x_i}$ across all $i \in [w]$. This construction and the verification of the uniqueness of clique counts~$\gamma(G_x)$ are presented in Section~\ref{sec:main_proof}.

\section{Constructing the Base Graphs}
\label{sec:Base_graph}

\begin{lemma}\label{q2}

For every non-negative integer $a$, and positive integers $\ell$ and $m$ satisfying $a + (\ell+3)2^\ell\leq m$, there exist a constant~ $C = C_{a,\ell,m}$ and a function~$f = f_{\ell,m}\colon \bigl[2^{2^\ell}\bigr]\to \mathbb{Z}_{\ge 0}$ such that for all $x
\in \bigl[2^{2^\ell}\bigr]$, there exists a graph $H_x$ on $m + 2\ell$ vertices satisfying the following properties.

\begin{enumerate}
    \item There exists a subset $M_0\subseteq V(H_x)$ of size $m$ that forms a clique, and 
    \item we have $\gamma(H_x)=C + 2^a (x-1) + 2^{a+2^\ell+1}f(x)$. \label{clique_number}
\end{enumerate}
\end{lemma}

\begin{proof}
Given $a$, $\ell$, and $m$ with $a + (\ell +3)2^\ell \leq m$, set $q \coloneqq 2^\ell$. Further, let $x \in [2^q]$ be given. We begin constructing the graph $H_x$ on $m+2\ell$ vertices by first partitioning the vertex set into disjoint sets $M_0, M_1, \dots, M_{\ell}$ having sizes $|M_0| = m$ and $|M_r| = 2$ for all $r \in [\ell]$. $M_0$ will be a large clique that forms the core, while the satellites $M_1,\dots,M_\ell$ will induce a complete $\ell$-partite graph with parts $M_1,\dots,M_\ell$. Our strategy is to control $\gamma(H_x)$ via the edges between $M_0$ and the satellites $M_1,\dots,M_\ell$. For each $r\in [\ell]$, we denote the two vertices in $M_r$ by $s_r(k)$ for $k \in \{0,1\}$. 

In order to define the edges between the sets~$M_0$ and~$V(H_x)\setminus M_0$, we first partition some vertices of~$M_0$ as follows. Let~$A \subseteq M_0$ be a vertex set of size~$a$. Then, for each $p\in [\ell]$, define $B_p \subseteq M_0$ to be a disjoint set of vertices of size $2^{p}$. The vertices of $B_p$ will be denoted by the coordinates $b_p(i,j)$ for all $i \in \{0,1\}$ and $j \in [2^{p-1}]$. Next, we let $D \subseteq M_0$ be a set of $2^\ell$ vertices indexed by the vectors in $\{0,1\}^\ell$. Finally, for each $p\in[\ell]$, let $E_p \subseteq M_0$ be disjoint sets of size $2^\ell$. Hence, we have that $$ A \sqcup B \sqcup D  \sqcup E \subseteq M_0,$$ where $B \coloneqq \bigsqcup_{\,p\in [\ell]} B_p$ and $E \coloneqq \bigsqcup_{\,p\in [\ell]} E_p$. Note that this partition of some vertices of~$M_0$ into these disjoint sets is possible as we have $$|A| + |B| + |D| + |E| = a + \sum_{p=1}^{\ell} 2^p  + (\ell +1)2^\ell < a + (\ell+3)2^{\ell} \leq m = |M_0|.$$

Now, given $x \in [2^{q}]$, we complete the construction of the graph $H_x$ by additionally adding edges between the sets $M_0$ and $V(H_x) \setminus M_0$. Let us fix a vertex $s_r(k) \in M_r$ for some $r \in [\ell]$ and $k \in \{0,1\}$. This vertex $s_r(k)$ will be connected to the following vertices of $M_0$. 

First, we connect the vertex~$s_r(k)$ to all vertices of the sets $A$, $E \setminus E_r$, and $B \setminus B_r$. Next, if~$k = 1$, we connect~$s_r(k)$ to the set of vertices $\{b_r(0,j): j \in [2^{r-1}]\} \subseteq B_r$; and if~$k =0$, then~$s_r(k)$ is not connected to any vertex of~$B_r$. Note that the vertex~$s_r(k)$ is adjacent to exactly~$k\cdot2^{r-1}$ vertices of~$B_r$ for each~$k \in \{0,1\}$ and~$r \in [\ell]$.

Finally, we need to define how $s_r(k)$ is connected to the vertices of~$D$. We wish to remark that the only dependency on $x$ in the graph $H_x$, and thereby in the clique count $\gamma(H_x)$, arises from the edges between the sets $D$ and $V(H_x)\setminus M_0$. For the given $x \in [2^{q}]$, we begin by writing $x-1$ as the binary sum $x - 1 = \sum_{j=1}^{q} c_j2^{j-1}$ where each $c_j\in \{0,1\}$, and setting $I_x=\{j\in[q]:c_{j}=1\}$. 

Note that~$I_x$ is a subset of~$[q] = [2^{\ell}]$. Moreover, by construction, the set~$D \subseteq M_0$ is a set of~$2^\ell$ vertices indexed by the set of vectors~$\{0,1\}^{\ell}$. Thus, we may define a bijection $\psi : D \simeq \{0,1\}^\ell \to [2^\ell]$ by $$\psi(\vv)\coloneqq1+\sum_{p\in[\ell]} \vv_{p}\, 2^{p-1} \quad \text{for all $\vv = (\vv_1, \ldots, \vv_{\ell}) \in D$.}$$ 
Define the set $D_x \coloneqq \psi^{-1}(I_x)$. Note that $D_x$ represents a set of vertices in $D$. Finally, for each~$r \in [\ell]$ and~$k \in \{0,1\}$, we connect the vertex~$s_r(k)$ to all vectors $\vv = (\vv_1, \ldots, \vv_\ell) \in D_x$ for which the $r$-th component~$\vv_r$ equals~$k$. This completes our construction of the graph~$H_x$. 

\begin{figure}[btp]
  \begin{subfigure}{0.47\textwidth}
  \centering
  \resizebox{\linewidth}{!}{%
  \begin{tikzpicture}[
  v2/.style={fill=black,minimum size=3pt,ellipse,inner sep=1pt},
  node distance=1.5cm, scale = 1]

\draw (0,-1) ellipse [y radius=1.4cm, x radius=2.4cm] ;
\node[font=\fontsize{17.28}{17.28}] at (-1,-1) {$M_0$};

\node[above, font=\fontsize{9}{9}] at (0,-.3) {$A$};
\node[above, font=\fontsize{9}{9}] at (0,-1) {$B$};
\node[above, font=\fontsize{9}{9}] at (0,-1.7) {$D$};
\node[above, font=\fontsize{9}{9}] at (0,-2.4) {$E$};

\node[v2, rotate around={65: (0,-4)}] (a1) at (0,4){};
\node[v2, rotate around={55: (0,-4)}] (a2) at (0,4){};

\node[v2, rotate around={-65: (0,-4)}] (b1) at (0,4){};
\node[v2, rotate around={-55: (0,-4)}] (b2) at (0,4){};

\node[v2, rotate around={35: (0,-4)}] (c1) at (0,4){};
\node[v2, rotate around={25: (0,-4)}] (c2) at (0,4){};

\node[v2, rotate around={3: (0,-4)}] at (0,4){};
\node[v2, rotate around={-15: (0,-4)}] at (0,4){};
\node[v2, rotate around={-33: (0,-4)}] at (0,4){};

\draw[rotate=60] (0,4) ellipse [y radius=.5cm, x radius=.9cm] ;
\node[rotate around={60: (0,-4.7)}] at (0,4.7) {$M_1$};

\draw[rotate=-60] (0,4) ellipse [y radius=.5cm, x radius=.9cm];
\node[rotate around={-60: (0,-4.7)}] at (0,4.7) {$M_\ell$};

\draw[rotate=30] (0,4) ellipse [y radius=.5cm, x radius=.9cm];
\node[rotate around={30: (0,-4.7)}] at (0,4.7) {$M_2$};

  \draw (a1) -- (b1) -- (c1) -- (a2) -- (b2) -- (c2) -- (a1) -- (b2) -- (c1) -- (a1);
  \draw (a2) -- (c2) -- (b2);
  \draw (a2) -- (b1) -- (c2);

  \draw[dashed]  (0,.4) -- (a2) -- (-2.4,-1);
  \draw[dashed]  (0,.4) -- (a1) -- (-2.4,-1);

  \draw[dashed]  (0,.4) -- (c2) -- (-2.4,-1);
  \draw[dashed]  (0,.4) -- (c1) -- (-2.4,-1);

  \draw[dashed] (0,.4) -- (b2) -- (2.4,-1);
  \draw[dashed]  (0,.4) -- (b1) -- (2.4,-1);

  \draw[dotted] (2.4,-1) -- (-2.4,-1);

  \draw[dotted] (2.05,-1.7) -- (-2.05,-1.7);
  \draw[dotted] (2.05,-.3) -- (-2.05,-.3);
    
\end{tikzpicture}
} 
\vspace{-5pt}
\caption{Partition of the vertices}
\end{subfigure}
\hfill
\begin{subfigure}{0.51\textwidth}
\centering
\resizebox{\linewidth}{!}{%
\begin{tikzpicture}[v2/.style={fill=black,minimum size=3pt,ellipse,inner sep=1pt}, node distance=1.5cm, scale=1]

\node[rotate around = {-40:(0,5)}] (a) at (0,-1) {};
\draw[fill=violet!50] (a) circle [radius=1cm];
\node at (a) {$A$};

\node[rotate around = {55:(0,5)}] (b) at (0,-1) {};
\draw[fill=violet!50] (b) circle [radius=1cm];
\node at (b) {$B\setminus B_r$};

\node[rotate around = {-65:(0,5)}] (e) at (0,-1) {};
\draw[fill=violet!50] (e) circle [radius=1cm];
\node at (e) {$E\setminus E_r$};

\node[rotate around = {90:(0,2.5)}] (b2) at (0,1.5) {};
\node at ($(b2)+(.5,.7)$) {$B_r$};
\draw[fill=red!50] (b2) rectangle ($(b2)+(1,.5)$);
\draw (b2) rectangle ($(b2)+(1,-.5)$);

\node[rotate around = {90:(1.5,1.5)}] (e2) at (0,-1) {};
\node at (e2) {$ E_r$};
\draw (e2) circle [radius=1cm];

\node (d) at (0,-1) {};
\node at (0,-1.75) {$D\setminus D_x$};
\draw ($(d)-(1.5,1.5)$) rectangle ($(d)+(1.5,0)$);

\draw[fill=red!50] (d) rectangle ($(d)+(1.5,1.5)$);
\draw[fill=blue!50] (d) rectangle ($(d)+(-1.5,1.5)$);

\node at ($(d)+(0,.75)$) {$D_x$};

\node[v2, color=blue, rotate around={7: (0,-4)}] (a0) at (0,4){};
\node[above] at (a0) {$s_r(0)$};
\node[v2, color=red, rotate around={-7: (0,-4)}] (a1) at (0,4){};
\node[above] at (a1) {$s_r(1)$};

\draw (0,4) ellipse [y radius=.75cm, x radius=1.1cm] ;
\node at (0,4.95) {$M_r$};

\coordinate (et) at ($(e) + (0,1)$);
\coordinate (es) at ($(e) + (.71,-.71)$);

\coordinate (at) at ($(a) + (-.53,.85)$);
\coordinate (as) at ($(a) + (.95,-.31)$);

\coordinate (bt) at ($(b) + (.31,.95)$);
\coordinate (bs) at ($(b) + (-.85,-.53)$);

\draw[color=blue, dashed] (et) -- (a0) -- (es);
\draw[color=blue, dashed] (at)-- (a0) -- (as);
\draw[color=blue, dashed] (bt) -- (a0) -- (bs);

\draw[color=red, dashed] (et) -- (a1) -- (es);
\draw[color=red, dashed] (at) -- (a1) -- (as);
\draw[color=red, dashed] (bt) -- (a1) -- (bs);

\draw[color=red, dashed] ($(b2)+(0,0)$) -- (a1) -- ($(b2)+(0,.5)$);

\draw[color=blue, dashed] (-1.5,.5) -- (a0) -- (0,.5);
\draw[color=red, dashed] (0,.5) -- (a1) -- (1.5,.5);

\end{tikzpicture}
}
\vspace{5pt}
\caption{Adjacencies between $M_r$ and $M_0$ for $r \in [\ell]$}
\end{subfigure}
\caption{Structure of the graph~$H_x$.}
\end{figure}
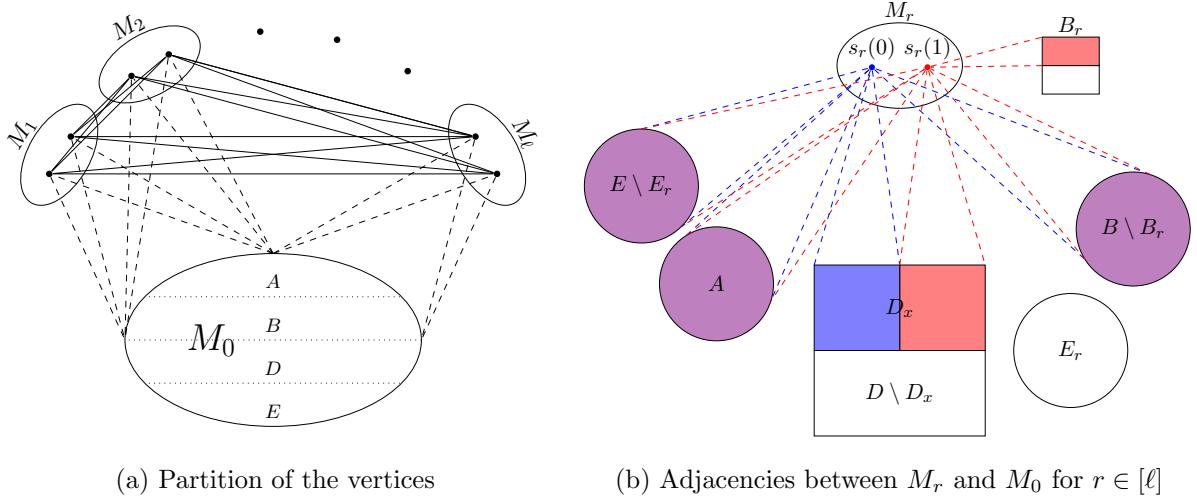

It is useful to note that 
\begin{equation}
    x-1=\sum_{j\in I_x} 2^{j-1} = \sum_{\vv \in D_x}2^{\psi(\vv)-1} = \sum_{\vv \in D_x} 2 ^{\sum_{p \in [\ell]} \vv_p \cdot 2^{p-1}}\,. \label{eq1}
\end{equation}

It now remains to show part \ref{clique_number} of the lemma. Observe that any clique in~$H_x$ contains at most one vertex from the set~$M_p$ for any~$p \in [\ell]$. Hence, we may compute~$\gamma(H_x)$ as follows. First, note that there are~$2^{m}$ cliques in~$M_0$. Next, let~$S$ be any non-empty subset of~$[\ell]$ and let~$\vv_S$ be any vector in the set~$\{0,1\}^S$. Here, the components of~$\vv_S$ are indexed by elements of the set~$S$. Observe that~$\vv_S$ represents a transversal of the sets $\{M_p : p \in S\}$, where for each~$r \in S$ and~$k \in \{0,1\}$, the vertex~$s_r(k)$ is chosen if the~$r$-th component of~$\vv_S$ satisfies~$(\vv_S)_r = k$. Let this transversal set of vertices be denoted by $M(\vv_S)$. 

Then, for each non-empty subset $S\subseteq [\ell]$ and for each vector $\vv_S \in \{0,1\}^S$, we define the set~$N^*(\vv_S) \coloneqq M_0\cap \bigl(\,\bigcap_{z\in M(\vv_S)} N(z)\bigr)$ to be the set of common neighbours in $M_0$ of the vertices in $M(\vv_S)$. Then each such $\vv_S$ adds an additional $2^{|N^*(\vv_S)|}$ number of cliques to the count of~$\gamma(H_x)$. Hence we have 
\begin{equation}
    \gamma(H_x) = 2^{m}+\sum _{\emptyset \neq S \subseteq [\ell]}\:\: \smashoperator{\sum_{\mkern 40mu\vv_S\in \{0,1\}^S}}\:\:\; 2^{\,|N^*(\vv_S)|}\,.\label{eq2}
\end{equation}

Now, fixing a non-empty subset $S \subseteq [\ell]$ and $\vv_S \in \{0,1\}^S$, let us compute $|N^*(\vv_S)|$. For this, we define the following function. Given $S$ and $\vv_S$ as above, let $$\sigma_x(\vv_S) \coloneqq \bigl| \{\textbf{u} \in D_x : \textbf{u}_p = (\vv_S)_p\, \text{ for all } p\in S\,\}\bigr|.$$
In other words, $\sigma_x(\vv_S)$ counts the number of vectors in $D_x$ whose \emph{restriction} to the set $S$ is given by the vector $\vv_S$. Note that by construction of $H_x$, the quantity $\sigma_x(\vv_S)$ counts the number of common neighbours in $D$ of the transversal set of vertices $M(\vv_S)$. Then, the size of the set~$N^*(\vv_S)$ can be computed as follows. 
\begin{eqnarray*}
    |N^*(\vv_S)| &=&  a + \biggl(\,\sum_{p\in S^c} 2^{p} + \sum_{p\in S} (\vv_S)_p\cdot 2^{p-1}\biggr) + \sigma_x(\vv_S) +  \sum_{p\in S^c} 2^\ell\\ &=& a + \sum_{p\in S^c}\bigl(2^\ell+2^p\bigr) + \sigma_x(\vv_S)+\sum_{p\in S}(\vv_S)_p\cdot 2^{p-1}.
\end{eqnarray*}
Here, in the first equality, the first term comes from the constant set $A$; the second term counts the common neighbours in each of the sets $B_p$, where we either have the entirety of $B_p$ if $p\notin S$ or a set of $(\vv_S)_p \cdot 2^{p-1}$ vertices of $B_p$ if $p \in S$; the third term follows from the definition of $\sigma_x$ above; and the last term results from the fact that the common neighbours of $M(\vv_S)$ in $E$ are given by the set $\bigcap_{p \in S} E \setminus E_p = \bigcup_{p \notin S} E_p$. 

When $S$ is a proper subset of $[\ell]$, then for any vector $\vv_S \in \{0,1\}^S$, we have that $|N^*(\vv_S)| \geq a + 2^\ell + 1 = a + q + 1$. Thus, we define the function $f = f_{\ell,m}: [2^q] \to \mathbb{Z}_{\ge 0}$, as required by the statement of the lemma, by the rule 
\begin{equation}f(x) \coloneqq \frac{1}{2^{(a+q+1)}} \cdot\sum _{\emptyset \neq S \subsetneq [\ell]}\:\: \smashoperator{\sum_{\mkern 40mu\vv_S\in \{0,1\}^S}}\:\:\; 2^{\,|N^*(\vv_S)|}\, \quad \text{ for all $x \in [2^q]$}.\label{eq3} \end{equation}

Note that by the computations above, $f(x)$ is integral and does not depend on the quantity~$a$. Now consider the case when $S = [\ell]$ and let $\vv$ be any vector in $\{0,1\}^{\ell}$. Note that in this case, for the quantity $\sigma_x(\vv)$, we have that $\sigma_x(\vv) = 1$ if $\vv \in D_x$, and $\sigma_x(\vv) = 0$ otherwise. Hence we have $|N^*(\vv)| = a + \sigma_x(\vv) + \sum_{p \in [\ell]} \vv_p\cdot 2^{p-1}$ and so the terms corresponding to $S = [\ell]$ in equation~\eqref{eq2} are given by
\begin{equation}
\begin{split}
    \sum_{\vv \in \{0,1\}^\ell}\!\! 2^{|N^*(\vv)|}\, &= \sum_{\vv \notin D_x} 2^{a + \sum_{p \in [\ell]} \vv_p 2^{p-1}} + \sum_{\vv \in D_x} 2^{a + 1+  \sum_{p \in [\ell]} \vv_p 2^{p-1}}\\ &= 2^a\!\!\sum_{\vv \in \{0,1\}^\ell} 2^{\sum_{p \in [\ell]} \vv_p 2^{p-1}} + 2^a \sum_{\vv \in D_x} 2^{\sum_{p \in [\ell]} \vv_p 2^{p-1}} =\, C^* + 2^a(x-1), \label{eq4}
    \end{split}
\end{equation}
where the second term of the last equation follows from equation~\eqref{eq1} and $C^*$ is defined by the following expression that is independent of $x$. Let $$C^* \coloneqq 2^a\cdot\!\!\sum_{\vv \in \{0,1\}^\ell} 2^{\sum_{p \in [\ell]} \vv_p 2^{p-1}}.$$

Finally, we define the required constant $C = C_{a,\ell,m}$ by $C\coloneqq 2^{m}+C^*$. Then, the lemma follows by substituting equations~\eqref{eq3} and \eqref{eq4} into equation~\eqref{eq2} to return  
$$\gamma(H_x) = 2^m + \sum_{\vv \in \{0,1\}^\ell}\!\! 2^{|N^*(\vv)|} + \sum _{\emptyset \neq S \subsetneq [\ell]}\:\: \smashoperator{\sum_{\mkern 40mu\vv_S\in \{0,1\}^S}}\:\; 2^{\,|N^*(\vv_S)|}\, =\, C + 2^a(x-1) + 2^{a+q+1}f(x).$$
\end{proof}

\section{Proof of Theorem~\ref{thm:main}}
\label{sec:main_proof}

In this section, we take the clique sum of the base graphs to prove Theorem~\ref{thm:main}. More precisely, we prove the following technical theorem which immediately implies Theorem~\ref{thm:main}. 

\begin{theorem}
    For all sufficiently large $n$, define $m = n-\lfloor 4\ln(2)\log^3n\rfloor +8$. Then for each integer~$x \in \bigl[2^{m-8}\bigr]$, there exists an $n$-vertex graph~$G_x$ such that for all distinct~$x, y\in \bigl[2^{m-8}\bigr]$, we have~$\gamma(G_x)\neq \gamma(G_y)$. 
\end{theorem}

\begin{proof}
    \setcounter{equation}{0}
    Let $n, m$ be given as in the statement of the theorem. We begin by constructing the intervals $[a_1,a_2), [a_2,a_3),\dots,[a_w,a_{w+1})$ as follows. Set $a_1\coloneqq0$. Having defined $a_i\le m-8$, let $\ell_i$ be the largest positive integer satisfying $a_{i}+(\ell_i+3)2^{\ell_i}\leq m$, and set     
    $a_{i+1} \coloneqq a_{i}+2^{\ell_i}$. Stop at the first $w$ for  which $a_{w+1}>m-8$. Since $a_i\le m-8$, the choice $\ell=1$ is admissible, so $\ell_i$ is well-defined and satisfies $\ell_i \geq 1$. Also, trivially  $\ell_i\leq \log m$. We shall use~$\Lambda_i\coloneqq a_{i+1}-a_i = 2^{\ell_i}$ to denote the length of the interval $[a_i, a_{i+1})$ for each $i \in [w]$.

    We construct $G_x$ for every $x\in[2^{a_{w+1}}]$.
Since $a_{w+1}>m-8$, restricting to
$x\in[2^{m-8}]$ proves the theorem. Moreover, each integer $x \in [2^{a_{w+1}}]$ can be uniquely associated with a $w$-tuple $(x_1, x_2, \dots, x_w)$ such that 
    \begin{equation}
        x-1=\sum_{i \in [w]} 2^{a_i}(x_i-1)\quad \text{and} \quad x_i\in \bigl[2^{a_{i+1}-a_i}\bigr] = \bigl[2^{\Lambda_i}\bigr] \text{ for each $i \in [w]$.}\label{eqTH1}
    \end{equation} 
    
    We construct the required $n$-vertex graph~$G_x$ as follows. Let $(x_1, x_2, \dots, x_w)$ be the unique $w$-tuple associated with~$x$, as described above. Note that by construction of the intervals~$[a_i, a_{i+1})$, we have that $a_i + (\ell_i + 3)2^{\ell_i} \leq m$ for each~$i \in [w]$. Hence, for each~$i \in [w]$, \Cref{q2} (with parameters~$a_i, \ell_i,m$) yields a constant~$C_i$ and a function~$f_i:\bigl[2^{\Lambda_i}\bigr]\to \mathbb{Z}_{\ge0}$ such that for the given~$x_i\in \bigl[2^{\Lambda_i}\bigr]$, there exists an $(m+2\ell_i)$-vertex graph~$H_{x_i}$ with $ \gamma(H_{x_i}) = C_i+2^{a_i}(x_i-1)+2^{a_i+\Lambda_i+1}f_i(x_i)$.

    Recall that for each $i \in [w]$, the graph $H_{x_i}$ contains a subset $M^i_0\subseteq V(H_{x_i})$ of size $m$ that forms a clique in $H_{x_i}$. We let $H_x$ be the clique sum of the graphs $H_{x_1}, \dots, H_{x_w}$ with respect to the cliques $M_0^1, \dots, M_0^w$. Observe that the graph~$H_x$ has exactly  $N \coloneqq m+2\sum_{i=1}^w \ell_i$ vertices, and the quantity~$N$ is independent of the choice of~$x$. We shall prove in the latter half of this proof that $N \leq n$. Assuming this inequality, the required $n$-vertex graph~$G_x$ is obtained from $H_x$ by additionally adding $n-N$ isolated vertices to $H_x$. 

    Next, we compute the number of cliques~$\gamma(G_x)$ in the graph~$G_x$. Observe that forming this clique sum of the graphs~$H_{x_1}, \dots, H_{x_w}$ does not create any new cliques. However, since the clique sum is constructed by superimposing the cliques~$M_0^1, \dots, M_0^w$ onto a common clique~$M_0 \subseteq V(H_x)$ of size~$m$, each of the~$2^m$ cliques within~$M_0$ is overcounted precisely $w-1$ times. So, it follows that $\gamma(H_x) = \sum_{i =1}^w \gamma(H_{x_i}) - (w-1)2^m$. Including the $n -N$ cliques induced by the set of isolated vertices~$V(G_x) \setminus  V(H_x)$, and by equation~\eqref{eqTH1}, we have that
    $$\gamma(G_x) = \sum_{i\leq w} \gamma(H_{x_i})  - (w-1)2^m + n-N = C+x+\sum_{i\leq w}2^{1+a_{i+1}}f_i(x_i),$$
    where we define $C\coloneqq \sum_{i\leq w} C_i - (w-1)2^m + n - N - 1$. We now show that the mapping~$x \mapsto \gamma(G_x)$ on the domain~$[2^{a_{w+1}}]$ is an injection.
    
    \begin{claim}\label{clm:injection}
        $\gamma(G_x)\neq \gamma(G_y)$ for all distinct $x,y\in [2^{a_{w+1}}]$.
    \end{claim}
    
    \begin{claimproof}
        Let distinct $x, y \in [2^{a_{w+1}}]$ be given. Let $(x_1, \dots, x_w)$ and $(y_1, \dots, y_w)$ be the unique $w$-tuples satisfying equation~\eqref{eqTH1}  for~$x$ and~$y$, respectively. Finally, let~$k \in [w]$ be the smallest integer for which we have~$x_k \neq y_k$. Note that such an integer~$k$ exists by the assumption~$x \neq y$. Now, evaluating $\gamma(G_z)$ modulo $2^{a_{k+1}}$ for $z \in \{x, y\}$, we obtain 
        \begin{equation}
            \gamma (G_z) \bmod{2^{a_{k+1}}} = \Biggl[\, C + 1 + \sum_{i = 1}^k 2^{a_i}(z_i -1) + \sum_{i = 1}^{k-1} 2^{1 + a_{i+1}} f_i(z_i) \Biggr] \bmod{2^{a_{k+1}}}. \label{eqTH2}
        \end{equation}

        By the minimality of $k$, we have that $x_i = y_i$ for all $i \in [k-1]$; and consequently, it follows that~$f_i(x_i) = f_i(y_i)$ for all $i \in [k-1]$. Thus, by equation~\eqref{eqTH2}, we have 
        \begin{align*}
            \bigl(\gamma(G_x) - \gamma(G_y) \bigr) \bmod 2^{a_{k+1}}&=\Biggl[\,\sum_{i = 1}^k 2^{a_i}(x_i -y_i) + \sum_{i = 1}^{k-1} 2^{1 + a_{i+1}} \bigl(f_i(x_i) - f_i(y_i)\bigr) \Biggr] \bmod 2^{a_{k+1}}\\[0.5em]
            &= 2^{a_k} (x_k - y_k) \bmod 2^{a_{k+1}}\, \neq\, 0.
        \end{align*}
        
        Here, the last relation follows from the fact that~$x_k \neq y_k$ and that $\bigl|2^{a_k}(x_k -y_k)\bigr| < 2^{a_{k+1}}$ since we have~$x_k , y_k \in \bigl[2^{a_{k+1} - a_k}\bigr]$. Thus, as $\gamma(G_x) \bmod{2^{a_{k+1}}} \neq \gamma(G_y) \bmod{2^{a_{k+1}}}$, it follows that~$\gamma(G_x) \neq \gamma(G_y)$, concluding the proof of the claim. 
    \end{claimproof}
     
     As $a_{w+1} > m - 8$ by definition, Claim~\ref{clm:injection} implies that $|\Gamma(n)| \geq 2^{m-8}=2^{n-\lfloor4\ln(2)\log^3 n\rfloor}\ge 2^{n-4\ln(2)\log^3 n}$, as witnessed by the collection of $n$-vertex graphs $\bigl\{G_x : x \in [2^{a_{w+1}}] \bigr\}$. However, it remains to show that the number of vertices in the graph $H_x$ for $x \in [2^{a_{w+1}}]$, as given by $N = m + 2\sum_{i = 1}^w \ell_i$, satisfies $N \leq n$. For this, we observe that the inequality $\ell_i \leq \log m$ for each $i \in [w]$ implies that $N \leq m + 2w \log{m}$, and hence, it suffices to prove that $m + 2w\log m \leq n$ for sufficiently large $n$. 
     
     We begin by obtaining an upper bound on the number of intervals~$w$. Let us do this by considering a different sequence~$b_1, \dots, b_{w'}$, defined using a non-integer relaxation of the recurrence relation used to define the sequence~$a_1, \dots, a_w$. More precisely, set~$b_1 \coloneqq 0$. For each~$i \geq 1$, let~$\ell_i'$ be the (not necessarily integer) solution to the equation $b_i+(2\log m)2^{\ell_i'}=m$, and set~$b_{i+1} \coloneqq b_i+2^{\ell'_i}$. We let~$w'$ be the smallest integer such that~$b_{w'+1}> m-8$. We first prove that~$w\leq w'$, and then we obtain an upper bound on~$w'$.

     \begin{claim}
     \label{clm:wwprime_reln}
     For sufficiently large~$n$, we have that $w \leq w'$.
     \end{claim}

     \begin{claimproof}

     Let $n$ be large enough so that $\log m > 17$. We first prove that $a_k \geq b_k$ for each $k \in [w+1]$. This is done by induction, starting with the base case $a_1 = b_1 = 0$. Suppose that $a_i\geq b_i$ for some $i\leq w$. If we have $2^{\ell_i} \geq 2^{\ell_i'}$, then the inductive hypothesis implies that $$a_{i+1}=a_i +2^{\ell_i}\geq b_i+2^{\ell_i'}=b_{i+1},$$
     as required. Hence, we may assume that $2^{\ell_i} < 2^{\ell_i'}$. 
     
     Recall that $\ell_i$ is the largest positive integer for which we have $a_i + (\ell_i + 3)2^{\ell_i} \leq m$. It is easy to see that $\ell_i\leq \log m-4$, as otherwise we obtain the inequality  
     $$ a_i+(\ell_i+3)2^{\ell_i}\geq  (\log (m)-1) \frac m {16} > m,$$
     thereby contradicting the definition of $\ell_i$. 
     
     The maximality of $\ell_i$ in the above definition also implies that $a_i + (\ell_i + 4)2^{\ell_i+1} > m$. Combining this fact with the recurrence relation for $a_{i+1}$ yields the following inequality.
     \begin{equation}
         a_{i+1} = a_i + 2^{\ell_i} > m - (2\ell_i + 8)2^{\ell_i} + 2^{\ell_i} = m - (2\ell_i + 7)2^{\ell_i}.
         \label{eqTH3}
     \end{equation}
     Similarly, the definition of $\ell_i'$ and the recurrence relation for $b_{i+1}$ results in the relation
     \begin{equation}
         b_{i+1} = b_i + 2^{\ell_i'} = m - (2\log m)2^{\ell_i'} + 2^{\ell_i'}= m - (2\log m-1)2^{\ell_i'}.
         \label{eqTH4}
     \end{equation}
     
     Now, the inequality $a_{i+1} \geq b_{i+1}$ follows by combining equations \eqref{eqTH3} and \eqref{eqTH4} to obtain $$a_{i+1}-b_{i+1}>(2\log m-1)2^{\ell_i'}-(2\ell_i+7)2^{\ell_i}>(2\log m-2\ell_i-8)2^{\ell_i'} \geq 0,$$
     where, for the second and third inequalities, we make use of the assumption $2^{\ell_i} < 2^{\ell_i'}$ and the fact $\ell_i\leq \log m-4$, respectively. 
     
     Thus, we have $a_k \geq b_k$ for all $k \in [w+1]$. The claim now follows. Indeed, if $w \geq w' + 1$, then we must have $a_w \geq a_{w' +1} \geq b_{w'+1} > m-8$, thereby contradicting the definition of $w$.
     \end{claimproof}
    
     Let us now obtain an upper bound on~$w'$. Set~$q \coloneqq 2\log m$. Then, by the definition of~$\ell_i'$, we have that~$2^{\ell_i'} = (m - b_i)/q$. Substituting this in the equation~$b_{i+1} = b_i + 2^{\ell_i'}$, we obtain the recurrence relation $b_{i+1}=\mu b_i + m/q$, where $\mu \coloneqq (q-1)/q$.
     Note that any sequence $\{s_i\}_{i \in \mathbb{N}}$ satisfying the recurrence relation $s_{i+1} = \alpha s_i + \beta$ with $\alpha \neq 1$ and $s_1 = 0$ has a closed form given by $s_{k+1} = \beta (1 - \alpha^k)/(1 -\alpha)$. Hence, it follows that 
     $$b_{w'}=\frac{m}{q}\cdot \frac{1-{\mu}^{w'-1}}{1-\mu} = m \bigl(1 - \mu^{w' - 1}\bigr).$$ 
     
     As $b_{w'} \leq m - 8$ by definition of $w'$, we have that $m \mu^{w' -1} \geq 8$, or equivalently, $(1/\mu)^{w' - 1} \leq m/8$. By taking natural logarithm and solving for $w'$, we obtain the inequality 
     \begin{eqnarray*}
     w'-1 &\leq& \frac{\ln(m/8)}{\ln(1/\mu)} = \frac{\ln (m/8)}{\ln\bigl(1 + 1/(q-1)\bigr)} \leq \frac{1+1/(q-1)}{1/(q-1)}\cdot \ln( m/8)\\
     &=& q \cdot \ln(2) \log(m/8)= 2\ln(2)\log^2(m)-2\ln(2)\log(8)\log(m)\\[0.1em]
     &\leq& 2 \ln(2) \log^2 (m) - 2,    
     \end{eqnarray*}
    where, the second inequality uses the fact that $\ln(1+x) \geq x/(1+x)$ for all $x > -1$. Thus, along with Claim~\ref{clm:wwprime_reln}, we obtain the upper bound $w\leq 2\ln(2)\log^2(m)  -1 $.
    The proof of Theorem~\ref{thm:main} now follows by observing that for sufficiently large~$n$, we have $$N \leq m+2w\log(m) \leq  n-4\ln(2)\log^3(n) +9 + 4\ln(2)\log^3(m)-2\log(m) \leq n.$$\end{proof}

\section{Concluding Remarks} 

In this paper, we proved that the asymptotic size of the set of clique counts~$\Gamma(n)$ is at least $2^{n-4\ln(2)\log^3(n)}$. However, this improvement does not fully resolve the problem. More precisely, our new lower bound and the existing upper bound together imply that $$n - 4 \ln(2) \log^3(n) \leq \log(|\Gamma(n)|) \leq n - 0.4 \log(n); $$ and it is still an open problem to close this gap and determine the exact asymptotics of the size of $\Gamma(n)$. The most intriguing question is what the power of $\log(n)$ should be here --- it is somewhere between $1$ and $3$, and we do not have a good guess as to where the truth lies. 

Beyond cliques, similar counting problems have also been explored for spanning trees and the set of cycle lengths. Sedl\'{a}\v{c}ek~\cite{Spanning_tree_orig} first explored the number of distinct spanning tree counts achievable by $n$-vertex graphs and proved some initial bounds. This was later improved by several authors. We refer the interested reader to~\cite{Spanning_tree_Alon} for a comprehensive overview of the subsequent advancements and the current best-known bounds.  

\paragraph{Acknowledgment.} The work leading up to this paper began at the Staycation Workshop hosted by LSE in September 2025. We are grateful to Leo Versteegen for organising the event.

\paragraph{Statement of AI use.} No AI tools were used to prove the results in this paper.

\bibliographystyle{siam}
\bibliography{counting_cliques}

\end{document}